\documentclass[11pt,twoside,a4paper]{article}

\usepackage{amssymb}
\usepackage{amsmath}
\usepackage{amsthm}

\allowdisplaybreaks

\def\rr{{\mathbb R}}
\def\rn{{{\rr}^n}}

\def\nn{{\mathbb N}}

\def\fz{\infty}
\def\az{\alpha}
\def\supp{{\mathop\mathrm{\,supp\,}}}

\def\lz{\lambda}

\def\ez{\epsilon}

\def\bz{\beta}

\def\wz{\widetilde}

\def\ls{\lesssim}
\def\gs{\gtrsim}

\def\tbz{{\triangle_\lz}}
\def\dmz{{dm_\lz}}

\def\riz{{R_{\Delta_\lz}}}

\def\rrp{{{\mathbb\rr}_+}}
\def\bmoz{{{\rm BMO}(\mathbb\rrp,\, dm_\lz)}}

\def\inzf{{\int_0^\fz}}
\def\szlzm{{(\sin\theta)^{2\lz-1}}}
\def\szlzp{{(\sin\theta)^{2\lz+1}}}
\def\szlzmx{{(\sin\theta)^{2\lz-1}\,d\theta}}

\def\wriz{{\wz\riz}}
\def\xtz{{x^{2\lz}\,dx}}
\def\mxzr{{m_\lz(I(x_0,\,r))}}
\def\myzr{{m_\lz(I(y_0,\,r))}}
\def\mxzxy{{m_\lz(I(x_0, |x_0-y_0|))}}

\def\lrz{{L^r(\rrp,\, dm_\lz)}}

\def\loz{{L^1(\rrp,\, dm_\lz)}}
\def\lpz{{L^p(\rrp,\, dm_\lz)}}
\def\lppz{{L^{p'}(\rrp,\, dm_\lz)}}
\def\linz{{L^\fz(\rrp,\, dm_\lz)}}
\def\hoz{{H^1({\mathbb R}_+,\, dm_\lz)}}

\def\dint{\displaystyle\int}

\def\dfrac{\displaystyle\frac}

\def\r{\right}
\def\lf{\left}

\newtheorem{thm}{Theorem}[section]
\newtheorem{lem}[thm]{Lemma}
\newtheorem{prop}[thm]{Proposition}
\newtheorem{rem}[thm]{Remark}
\newtheorem{cor}[thm]{Corollary}
\newtheorem{defn}[thm]{Definition}

\numberwithin{equation}{section}

\begin{document}

\arraycolsep=1pt

\title{\Large\bf  Factorization for Hardy spaces and characterization for BMO spaces via commutators
in the Bessel setting}
\author{Xuan Thinh Duong, Ji Li, Brett D. Wick, Dongyong Yang}

\date{}
\maketitle

\begin{center}
\begin{minipage}{13.5cm}\small

{\noindent  {\bf Abstract:}\  Fix $\lambda>0$. Consider  the Hardy space $H^1(\mathbb{R}_+,dm_\lambda)$ in the sense of Coifman and Weiss,
where $\mathbb{R_+}:=(0,\infty)$ and $dm_\lambda:=x^{2\lambda}dx$ with $dx$ the Lebesgue measure.
Also consider
the Bessel operators
$\triangle_\lambda:=-\frac{d^2}{dx^2}-\frac{2\lambda}{x}
\frac d{dx}$, and $S_\lambda:=-\frac{d^2}{dx^2}+\frac{\lambda^2-\lambda}{x^2}$ on $\mathbb{R_+}$.
The Hardy spaces $H^1_{\Delta_\lambda}$ and  $H^1_{S_\lambda}$ associated with $\Delta_\lambda$
and $S_\lambda$  are defined via the Riesz transforms $R_{\Delta_\lambda}:=\partial_x (\Delta_\lambda)^{-1/2}$
and $R_{S_\lambda}:= x^\lambda\partial_x x^{-\lambda} (S_\lambda)^{-1/2}$, respectively. It is
known that $H^1_{\Delta_\lambda}$ and $H^1(\mathbb{R}_+,dm_\lambda)$ coincide but they are different
from $H^1_{S_\lambda}$. In this article, we prove the following:
(a) a weak factorization of $H^1(\mathbb{R}_+,dm_\lambda)$ by using a bilinear form of the
Riesz transform $R_{\Delta_\lambda}$, which implies the characterization of the BMO space
associated to $\Delta_\lambda$ via the commutators related to $R_{\Delta_\lambda}$; (b) the BMO space associated to $S_\lambda$
can not be characterized by commutators related to $R_{S_\lambda}$, which implies that
$H^1_{S_\lambda}$ does not have a weak factorization via a bilinear form of the Riesz transform $R_{S_\lambda}$.

}

\end{minipage}
\end{center}

%
%
\bigskip

{ {\it Keywords}: BMO; commutator; Hardy space; factorization; Bessel operator; Riesz transform.}

\medskip

{{Mathematics Subject Classification 2010:} {42B30, 42B20, 42B35}}

\section{Introduction and statement of main results\label{s1}}

\noindent Recall that the classical Hardy space $H^p$, $0<p<\infty$, on the unit disc
$\mathbb{D}=\{ z\in\mathbb{C}:\ |z|<1\}$ is defined as the space of holomorphic functions $f=u+iv$, i.e., those
satisfying the Cauchy--Riemann equations
$ {\partial u\over \partial x} ={\partial v\over \partial y}\ \ {\rm and\ \ }
 {\partial u\over \partial y} =-{\partial v\over \partial x}\ \ {\rm in\ \ }\mathbb{D},   $
such that
$$  \|f\|_{H^p(\mathbb{D})}:=\sup_{0\leq r<1} \Big( {1\over 2\pi}\int_0^{2\pi}  \big|f(re^{it})\big|^p dt \Big)^{1\over p}<\infty.  $$
It is well known that the product of two
 $H^2(\mathbb{D})$ functions belongs to Hardy space $H^1(\mathbb{D})$, but in fact the converse is also true, and is known as
 {\it  the Riesz factorization theorem}:
 {\it ``A function $f$ is in $H^1(\mathbb{D})$ if and only if there exist $g,h\in H^2(\mathbb{D})$ with
 $f= g\cdot h$ and $\|f\|_{H^1(\mathbb{D})} =\|g\|_{H^2(\mathbb{D})}\|h\|_{H^2(\mathbb{D})}$.''
  }
 This factorization result plays an important role in studying function theory and operator
 theory connected to the spaces $H^1(\mathbb{D})$, $H^2(\mathbb{D})$ and the space $BMOA(\mathbb{D})$ (analytic BMO).
 \smallskip

The real-variable Hardy space theory on
$n$-dimensional Euclidean space $\rn$ ($n\geq1$) plays an important
role in harmonic analysis and has been systematically developed. We point out
that two closely related characterizations for $H^1(\mathbb{R}^n)$ are that:
(1) $H^1(\mathbb{R}^n)$ can be characterized in terms of Riesz transforms;  (2)  $H^1(\mathbb{R}^n)$
 can be viewed as the boundary of the Hardy space $H^1(\mathbb{R}^{n+1}_+)$ consisting of
 systems of conjugate harmonic functions $F=(u_0,u_1,\ldots,u_n)$, which
 satisfy the generalized Cauchy--Riemann equations $ \sum_{j=0}^n {\partial u_j\over \partial x_j} =0\ \ {\rm and\ \ }
 {\partial u_j\over \partial x_k} ={\partial u_k\over \partial x_j}\ \ {\rm in\ \ }\mathbb{R}_+^{n+1}$, $0\leq j,k\leq n$,
see \cite{fs72,sw60}. However, the analogue of the Riesz factorization theorem, sometimes referred to as strong factorization,
is not true for real-variable Hardy space $H^1(\mathbb{R}^n)$.  Nevertheless,
Coifman, Rochberg and Weiss \cite{crw} provided a suitable replacement that works in studying function
theory and operator theory of $H^1(\mathbb{R}^n)$, the weak factorization via a bilinear form related to the
Riesz transform (Hilbert transform in dimension 1).

The theory of the classical Hardy space is intimately connected to the Laplacian; changing
the differential operator introduces new challenges and directions to explore.
In 1965, Muckenhoupt and Stein in \cite{ms} introduced a notion of conjugacy associated with this Bessel operator $\tbz$,
which is defined by
\begin{equation*}
\tbz f(x):=-\frac{d^2}{dx^2}f(x)-\frac{2\lz}{x}\frac{d}{dx}f(x),\quad x>0.
\end{equation*}
They developed a theory in the setting of
$\tbz$ which parallels the classical one associated to $\triangle$.  Results on $\lpz$-boundedness of conjugate
functions and fractional integrals associated with $\tbz$ were
obtained, where $p\in[1, \fz)$, $\rr_+:=(0, \fz)$ and $\dmz(x):= x^{2\lz}\,dx$.
Since then, many problems based on the Bessel context were studied;
see, for example, \cite{ak,bcfr,bfbmt,bfs,bhnv,k78,v08,yy}. In particular,
the properties and $L^p$ boundedness $(1<p<\infty)$ of Riesz transforms
\begin{equation*}
R_{\Delta_\lambda}f := \partial_x (\Delta_\lz)^{-1/2}f
\end{equation*}
related to $\Delta_\lz$ have been studied in \cite{ak,bcfr,bfbmt,ms,v08}.
The related Hardy space
$$H^1(\mathbb{R}_+, \dmz):=\{ f\in L^1(\mathbb{R}_+, \dmz):\  R_{\Delta_\lambda}f\in L^1(\mathbb{R}_+, \dmz)\}$$
with norm $\|f\|_{H^1(\mathbb{R}_+,\,\dmz)}:=\|f\|_{L^1(\mathbb{R}_+,\,\dmz)}+\|R_{\Delta_\lambda}f\|_{L^1(\mathbb{R}_+,\,\dmz)}$ has been studied
by Betancor et al. in \cite{bdt}. For $f\in H^1(\mathbb{R}_+, \dmz)$, we have
that the pair of functions
$$ u(t,x):=P_t^{[\lambda]}(f)(x) \quad {\rm and}\quad  v(t,x):=Q_t^{[\lambda]}(f)(x), \quad t,x\in \mathbb{R}_+  $$
satisfy the following generalized Cauchy--Riemann equations
$$ {\partial u\over \partial x} =-{\partial v\over \partial t}\ \ {\rm and\ \ }
 {\partial u\over \partial t} ={\partial v\over \partial x} - {2\lambda\over x} v\ \ {\rm in\ \ }\mathbb{R}_+. $$
Here $P_t^{[\lambda]}(f)$ is the Poisson integral of $f$ with the Poisson semigroup $P_t^{[\lambda]}:=e^{-\sqrt{\Delta_\lambda}}$, and
$Q_t^{[\lambda]}(f)$ is the conjugate Poisson integral, see Section 2 for precise definitions.

\smallskip
Following a different procedure in \cite{ms},  the
Riesz transform
\begin{equation*}
R_{S_\lz}f := A_\lambda (S_\lz)^{-1/2}f, \quad  {\rm where}\ A_\lambda:= x^\lambda \partial_x x^{-\lambda}
\end{equation*}
has also been studied by Betancor et al. \cite{bdt}, which is related to the other Bessel operator $S_\lambda$
defined by
\begin{equation}\label{bessel 2}
S_\lambda f(x):=-\frac{d^2}{dx^2}f(x)+ \frac{\lz^2-\lz}{x^2}f(x),\quad x>0.
\end{equation}
Moreover, the corresponding Hardy space
$$H^1_{S_\lambda}(\mathbb{R}_+,\,dx):=\big\{ f\in L^1(\mathbb{R}_+,\,dx): R_{S_\lz}(f)\in L^1(\mathbb{R}_+,\,dx) \big\}$$
with norm $\|f\|_{H^1_{S_\lambda}(\mathbb{R}_+,\,dx)}= \|f\|_{L^1(\mathbb{R}_+,\,dx)} + \|R_{S_\lambda}f\|_{L^1(\mathbb{R}_+,\,dx)}$
was characterized. And the Poisson integral and the conjugate Poisson integral of the function $f\in H^1_{S_\lambda}(\mathbb{R}_+,\,dx)$ also satisfy a generalized
Cauchy--Riemann equations.

\medskip
A natural question is that: Do the Hardy spaces $H^1(\mathbb{R}_+, \dmz)$ and $H^1_{S_\lambda}(\mathbb{R}_+,\,dx)$ have Riesz type factorization or weak factorization in terms of a bilinear form related to $R_{\Delta_\lambda}$ and $R_{S_\lambda}$, respectively?

\medskip

The aim of this paper is twofold. We first build up a weak factorization for the Hardy space $H^1(\mathbb{R}_+, \dmz)$ in terms of a bilinear form related to $R_{\Delta_\lambda}$. Then we further prove that this weak factorization implies the characterization of the dual of $H^1(\mathbb{R}_+, \dmz)$ via commutators related to $R_{\Delta_\lambda}$. Second, we point out that a weak factorization for the Hardy space $H_{S_\lambda}^1(\mathbb{R}_+, dx)$ in terms of a bilinear form related to $R_{S_\lambda}$ is not true, by proving that the dual of $H_{S_\lambda}^1(\mathbb{R}_+, dx)$ can not be characterized by commutators related to $R_{S_\lambda}$.

To state our main results, we first recall some necessary notions and notation. Throughout this paper,
for any $x$, $r\in \mathbb{R}_+$, $I(x, r):=(x-r, x+r)\cap \mathbb{R}_+$.
 From the definition of the measure $m_\lambda$ (i.e., $dm_\lambda(x):=x^{2\lambda}dx$),  it is obvious that
 there exists a positive constant $C\in(1, \fz)$ such that for all $x,\,r\in\mathbb{R}_+$,
\begin{equation}\label{volume property-1}
 C^{-1}m_\lz(I(x, r))\le x^{2\lz}r+r^{2\lz+1}\le C m_\lz(I(x, r)).
\end{equation}
Thus  $(\mathbb{R}_+, \rho, dm_\lz)$ is a space of homogeneous type in the sense of Coifman and Weiss \cite{cw77},
where $\rho(x,y):=|x-y|$ for all $x,\,y\in\rrp$.

We now state our first main result on the weak factorization (via Riesz transform $\riz$ and its adjoint operator $\wriz$) of the Hardy space $\hoz$.
\begin{thm}\label{t-equiv character H1}
Let $p\in(1, \fz)$ and $p'$ be the conjugate of $p$. For any $f\in\hoz$, there exist numbers $\{\az^k_j\}_{k,\,j}$, functions $\{g^k_j\}_{k,\,j}\subset\lpz$ and $\{h^k_j\}_{k,\,j}\subset\lppz$ such that
\begin{equation}\label{represent of H1}
f=\sum_{k=1}^\fz\sum_{j=1}^\fz \az^k_j\,\Pi\lf(g^k_j,h^k_j\r)
\end{equation}
in $\hoz$, where the operator $\Pi$ is defined as follows: for $g\in \lpz$ and $h\in \lppz$,
\begin{equation}\label{def of pi}
\Pi(g,h):=g\riz h- h\wriz g.
\end{equation}
Moreover, there exists a positive constant $C$ independent of $f$ such that
\begin{eqnarray*}
C^{-1}\|f\|_\hoz&\le&\inf\lf\{\sum_{k=1}^\fz\sum_{j=1}^\fz\lf|\az^k_j\,\r|\lf\|g^k_j\r\|_\lpz\lf\|h^k_j\r\|_\lppz:\r.\\
&\quad&\lf.\quad \quad f=\sum_{k=1}^\fz\sum_{j=1}^\fz \az^k_j\,\Pi\lf(g^k_j,h^k_j\r)\r\}\le C\|f\|_\hoz.
\end{eqnarray*}
\end{thm}

Our second main result provides a characterization of the BMO space $\bmoz$, which is the dual of $\hoz$,   in terms of
the commutators adapted to the Riesz transform $\riz$.  Recall the definition of the BMO space associated with the Bessel operator, which is the dual space of $\hoz$.
\begin{defn}[\cite{yy}]\label{d-bmo}
A function $f\in L^1_{\rm loc}(\rrp,dm_\lambda)$ belongs to
the {\it space} $\bmoz$ if
\begin{equation*}
\sup_{x,\,r\in(0,\,\fz)}\frac1{m_\lz(I(x,\,r))}\dint_{I(x,\,r)}
\lf|f(y)-\frac1{m_\lz(I(x, r))}\dint_{I(x,\,r)}f(z)\,dm_\lz(z)\r|\,dm_\lz(y)<\fz.
\end{equation*}
\end{defn}

Suppose $b\in L^1_{\rm loc}(\mathbb{R}_+,\,dm_\lz)$ and $f\in \lpz$. Let $[b, \riz]$ be the commutator defined by
\begin{equation*}
[b, \riz]f(x):= b(x)\riz f(x)-\riz(bf)(x).
\end{equation*}

\begin{thm}\label{t-commutator character BMO}
Let $b\in \cup_{q>1}L^q_{\rm loc}(\mathbb{R}_+,\,dm_\lz)$ and $p\in(1, \fz)$.

$(1)$ If $b\in\bmoz$, then the commutator  $[b,\riz]$ is bounded on $\lpz$ with the operator norm
\begin{eqnarray*}
\left\|[b, \riz] \right\|_{\lpz\to\lpz} \le C\|b\|_\bmoz.
\end{eqnarray*}

$(2)$ If $[b,\riz]$ is bounded on $\lpz$, then $b\in\bmoz$ and
\begin{eqnarray*}
\|b\|_{\bmoz}&\le& C\left\|[b, \riz] \right\|_{\lpz\to\lpz}.
\end{eqnarray*}
\end{thm}

We will provide the proof
of Theorems \ref{t-equiv character H1} and \ref{t-commutator character BMO} in the
following structure: we first provide the proof of (1) in Theorem \ref{t-commutator character BMO},
 which plays the key role to prove Theorem \ref{t-equiv character H1}. Then, (2) in Theorem \ref{t-commutator character BMO}
 follows directly from
 the weak factorization  of
$\hoz$ in Theorem \ref{t-equiv character H1} via duality.

The next main result that we provide is that the BMO space BMO$_{S_\lambda}(\mathbb{R}_+, dx)$ associated with $S_\lambda$, which is the dual of $H^1_{S_\lambda}(\mathbb{R}_+,dx)$,  can not be characterized by the commutators with respect to the Riesz transform $R_{S_\lambda}$.
\begin{thm}\label{t-commutator not character BMO}
There exists a locally integrable function  $b\not\in {\rm BMO}_{S_\lambda}(\mathbb{R}_+,\,dx)$, such that for $1<p<\infty$,
 the commutator  $[b,R_{S_\lz}]$ is bounded on $L^p(\mathbb{R}_+,\,dx)$ with the operator norm
\begin{eqnarray*}
\left\|[b, R_{S_\lz}] \right\|_{L^p(\mathbb{R}_+,\,dx) \to L^p(\mathbb{R}_+,\,dx)} \le C_b,
\end{eqnarray*}
where $C_b$ is a positive constant related to the function $b$.
\end{thm}
As a consequence, we have the following argument.
\begin{cor}\label{c weak factorization}
 A weak factorization for $H^1_{S_\lambda}(\mathbb{R}_+,dx)$ in the form of Theorem \ref{t-equiv character H1} with respect to a bilinear form
realated to $R_{S_\lz}$ is not true.
\end{cor}
%
%
%
%

An outline of the paper is as follows.  In Section \ref{s2} we recall the Hardy spaces
 associated with $\Delta_\lambda$ and $S_{\lambda}$. Also we collect some fundamental
  estimates of the kernel of the Riesz transforms  $\riz$ and $R_{S_\lambda}$,
   especially the size estimate and the H\"older's regularity in Proposition \ref{t:RieszCZ}
    (Proposition \ref{p upper}), and the kernel lower bounds
in Proposition \ref{p-estimate of riesz kernel}
for $\riz$. In Section \ref{s:MainResult} we prove Theorems \ref{t-equiv character H1} and \ref{t-commutator character BMO}. We note that
our auxiliary result Lemma \ref{l-atomic estimate} is an important ingredient of the proof of
 Theorem \ref{t-equiv character H1}, an earlier analogue of which appears in \cite{lw} but with different proof.
In Section \ref{s:MainResult 2} we
prove Theorem \ref{t-commutator not character BMO} by providing a specific example of the locally integrable
function  $b$ from the classical BMO space ${\rm BMO}(\mathbb{R}_+,\,dx)$ but $b\not\in {\rm BMO}_{S_\lambda}(\mathbb{R}_+,dx)$.
 Then we prove Corollary \ref{c weak factorization}. 

\medskip
Throughout the paper,
we denote by $C$ and $\widetilde{C}$ {\it positive constants} which
are independent of the main parameters, but they may vary from line to
line. For every $p\in(1, \fz)$, we denote by $p'$ the conjugate of $p$, i.e., $\frac{1}{p'}+\frac{1}{p}=1$.  If $f\le Cg$, we then write $f\ls g$ or $g\gs f$;
and if $f \ls g\ls f$, we  write $f\sim g.$
For any $k\in \mathbb{R}_+$ and $I:= I(x, r)$ for some $x$, $r\in (0, \fz)$,
$kI:=I(x, kr)$.

\section{Hardy and BMO spaces,  Riesz transforms  associated with $\Delta_\lambda$ and $S_\lambda$}\label{s2}

In this section we recall the Hardy and BMO spaces, and some important properties of   and  Riesz transforms
related to the Bessel operator $\Delta_\lambda$ and $S_\lambda$  from \cite{ms,bdt,bfbmt,bfs}.

We now recall the atomic characterization of the Hardy spaces $\hoz$ in \cite{bdt}.

\begin{defn}[\cite{bdt}]\label{d-atomic H1}
 A function $a$ is called a $(1, \fz)_{\Delta_\lz}$-atom
if there exists an open bounded interval $I\subset \mathbb{R}_+$
such that $\supp(a)\subset I$, $\|a\|_\linz\le[m_\lz(I)]^{-1}$
and $\int_0^\fz a(x)\,\dmz(x)=0.$
\end{defn}

We point out that from \cite{bdt}, the Hardy space $\hoz$ can be characterized via atomic decomposition. That is,
an $\loz$ function $f\in \hoz$ if and only if
$$f=\sum_{j=1}^\fz\az_j a_j\quad {\rm in}\quad  \loz,$$ where for every $j$,
$a_j$ is a $(1, \fz)_{\Delta_\lz}$-atom and $\az_j\in\rr$ satisfying that
$\sum_{j=1}^\fz|\az_j|<\fz$. Moreover,
$\|f\|_\hoz\approx \inf\lf\{\sum_{j=1}^\fz|\az_j|\r\},$
where the infimum is taken over all the decompositions of $f$ as above.

We also note that $\hoz$  can also be
characterized in terms of the radial
maximal function associated with the Hankel convolution of a class
of functions, including the Poisson semigroup and the heat
semigroup as special cases.  It is also  proved in \cite{bdt} that $H^1(\mathbb{R}_+, \dmz)$ is the one associated with
the space of homogeneous type $(\mathbb{R}_+, \rho, dm_\lz)$ defined by Coifman and Weiss in \cite{cw77}.


Next we recall the Poisson integral, the conjugate Poisson integral and the properties of the Riesz transforms.
As in \cite{bdt}, let $\{P^{\lz}_t\}_{t>0}$ be the Poisson semigroup $\{e^{-t\sqrt{\Delta_\lz}}\}_{t>0}$ defined by
\begin{equation*}
P^{[\lz]}_tf(x):=\inzf P^{[\lz]}_t(x,y)f(y)y^{2\lz}\,dy,
\end{equation*}
where
\begin{equation*}
P^{[\lz]}_t(x,y)=\inzf e^{-tz}(xz)^{-\lz+1/2}J_{\lz-1/2}(xz)(yz)^{-\lz+1/2}J_{\lz-1/2}(yz)z^{2\lz}\,dz
\end{equation*}
and $J_\nu$ is the Bessel function of the first kind and order $\nu$.  Weinstein \cite{w48} established the following formula for $P^{[\lz]}_t(x,y)$: $t,\,x,\, y\in\mathbb{R}_+$,
\begin{equation*}
P^{[\lz]}_t(x,y)=\frac{2\lz t}{\pi}\int_0^\pi\frac{(\sin\theta)^{2\lz-1}}{(x^2+y^2+t^2-2xy\cos\theta)^{\lz+1}}\,d\theta.
\end{equation*}
The $\Delta_\lambda$-conjugate of the Poisson integral of $f$ is defined by
$$  Q^{[\lz]}_tf(x):=\inzf Q^{[\lz]}_t(x,y)f(y)y^{2\lz}\,dy,
 $$
where
\begin{equation*}
Q^{[\lz]}_t(x,y)=\frac{-2\lz }{\pi}\int_0^\pi\frac{ (x-y\cos\theta) (\sin\theta)^{2\lz-1}}{(x^2+y^2+t^2-2xy\cos\theta)^{\lz+1}}\,d\theta.
\end{equation*}
From this, we deduce that for any $x,\,y\in\rrp$,
\begin{eqnarray}\label{riesz kernel}
\riz(x, y)=\partial_x\inzf P^{[\lz]}_t(x,y)\,dt=-\dfrac{2\lz}{\pi}\dint_0^\pi\dfrac{(x-y\cos\theta)(\sin\theta)^{2\lz-1}}
{(x^2+y^2-2xy\cos\theta)^{\lz+1}}\,d\theta = \lim_{t\to0} Q^{[\lz]}_tf(x).\quad\quad
\end{eqnarray}

We note that, as indicated in \cite{bdt}, this Riesz transform $\riz$ is a Calder\'on--Zygmund operator (see also \cite{bfs}).
For the convenience of the readers we provide all the details of the verification here.
\begin{prop}
\label{t:RieszCZ}
The kernel $\riz(x,y)$ satisfies the following conditions:
\begin{itemize}
  \item [i)] for every $x,y\in\mathbb{R}_+$ with $x\not=y$,
  \begin{equation}\label{cz kernel condition-1}
  |\riz(x,y)|\ls \frac1{m_\lz(I(x, |x-y|))} ;
  \end{equation}
  \item [ii)] for every $x,\,x_0,\, y\in \mathbb{R}_+$ with $|x_0-x|<|x_0-y|/2$,
  \begin{eqnarray}\label{cz kernel condition-2}
   &&|\riz(y, x_0)-\riz(y,x)|+ |\riz(x_0, y)-\riz(x,y)|\nonumber\\
   &&\quad\ls \frac{|x_0-x|}{|x_0-y|}\frac1{m_\lz(I(x_0, |x_0-y|))}.
  \end{eqnarray}
\end{itemize}
\end{prop}

\begin{proof}
 Recall that
\begin{equation*}
\riz(x,y)=-\dfrac{2\lz}{\pi}\dint_0^\pi\dfrac{(x-y\cos\theta)(\sin\theta)^{2\lz-1}}
{(x^2+y^2-2xy\cos\theta)^{\lz+1}}\,d\theta.
\end{equation*}
We first verify \eqref{cz kernel condition-1}. Suppose $x,y\in\mathbb{R}_+$ with $x\not=y$. We now consider the following two cases.

Case 1, $x\le2|x-y|$. Note that  $|x-y\cos\theta|\leq |x^2+y^2-2xy\cos\theta|^{1\over2}$.
Combining the fact that
\begin{equation}\label{sin function integral}
\int_0^{\pi}(\sin\theta)^{2\lz-1}\,d\theta =2\int_0^{\pi\over2}(\sin\theta)^{2\lz-1}\,d\theta =\frac{\Gamma(\lz)\sqrt\pi}{\Gamma(\lz+1/2)}
\end{equation}
 and  the property \eqref{volume property-1} of the measure $dm_\lambda$, we obtain that
 $$
	 |\riz(x,y)|\ls \dint_0^\pi\dfrac{(\sin\theta)^{2\lz-1}}
	{|x-y|^{2\lz+1}}\,d\theta   \ls\dfrac1{|x-y|^{2\lz+1}}\sim \frac1{m_\lz(I(x, |x-y|))}.
$$

Case 2, $x>2|x-y|$. Note that in this case, $x/2\le y\le3x/2$.  Thus, by noting that $1-\cos \theta \geq 2\big( {\theta\over\pi} \big)^2$ for $\theta\in [0,\pi]$, we have
\begin{eqnarray*}
|\riz(x,y)|&\ls&\dint_0^{\pi}\dfrac{(\sin\theta)^{2\lz-1}}
{[(x-y)^2+2xy(1-\cos\theta)]^{\lz+{1\over2}}}\,d\theta \\
&\ls&\int_0^{\pi}\frac{\theta^{2\lz-1}}{[(x-y)^2+4xy\theta^2/\pi^2]^{\lz+{1\over2}}}\,d\theta
\\
&\ls&\frac1{|x-y|^{2\lz+1}}\frac{|x-y|^{2\lz}}{x^\lz y^\lz}\int_0^\fz\frac{\bz^{2\lz-1}}{[1+\bz^2]^{\lz+{1\over2}}}\,d\bz\\
&\ls&\frac1{|x-y|x^{2\lz}}\sim\frac1{m_\lz(I(x, |x-y|))}.
\end{eqnarray*}
Combining the estimates in Cases 1 and 2 we obtain \eqref{cz kernel condition-1}.

We turn to \eqref{cz kernel condition-2}. We point out that it suffices to prove that
when $|x_0-x|<|x_0-y|/2$,
\begin{equation}\label{cz kernel condition-3}
|\riz(x_0, y)-\riz(x,y)|\ls \frac{|x_0-x|}{|x_0-y|}\frac1{m_\lz(I(x_0, |x_0-y|))},
\end{equation}
then the estimate for $|\riz(y,x_0)-\riz(y,x)|$ follows similarly.

By the Mean Value Theorem, there exists $\xi:=tx_0+(1-t)x$ for some $t\in(0, 1)$ such that
$$\left|\riz(x_0, y)-\riz(x,y)\right|=|x_0-x||\partial_x\riz(\xi,y)|.$$
Observe that
\begin{eqnarray*}
|\partial_x\riz(\xi,y)|&\ls&\lf|\dint_0^\pi\dfrac{(\sin\theta)^{2\lz-1}}
{(\xi^2+y^2-2\xi y\cos\theta)^{\lz+1}}\,d\theta\r|+\lf|
\dint_0^\pi\dfrac{(\xi-y\cos\theta)^2(\sin\theta)^{2\lz-1}}
{(\xi^2+y^2-2\xi y\cos\theta)^{\lz+2}}\,d\theta\r|\\
&\ls&\dint_0^\pi\dfrac{(\sin\theta)^{2\lz-1}}
{(\xi^2+y^2-2\xi y\cos\theta)^{\lz+1}}\,d\theta\\
&\ls&\dint_0^{\pi/2}\dfrac{(\sin\theta)^{2\lz-1}}
{(\xi^2+y^2-2\xi y\cos\theta)^{\lz+1}}\,d\theta.
\end{eqnarray*}
To show \eqref{cz kernel condition-3}, it suffices to prove that
\begin{equation*}
\dint_0^{\pi/2}\dfrac{(\sin\theta)^{2\lz-1}}
{(\xi^2+y^2-2\xi y\cos\theta)^{\lz+1}}\,d\theta\ls \frac1{|x_0-y|}\frac1{m_\lz(I(x_0, |x_0-y|))}.
\end{equation*}
To see this, we first point out that from the definition of $\xi$,
\begin{equation}\label{distance estimate}
\frac12|x_0-y|\le |\xi-y|\le \frac32|x_0-y|.
\end{equation}
Then, we consider the following two cases.

\noindent Case 1, $x_0\le 2|x_0-y|$. It follows that
$$\dint_0^{\pi/2}\dfrac{(\sin\theta)^{2\lz-1}}
{(\xi^2+y^2-2\xi y\cos\theta)^{\lz+1}}\,d\theta\ls \int_0^{\pi/2}\frac{\szlzm}{(\xi-y)^{2\lz+2}}\,d\theta\ls \frac1{(x_0-y)^{2\lz+2}};$$
Case 2, $x_0>2|x_0-y|$. Note that in this case $\xi\sim y\sim x_0$. By \eqref{distance estimate}, we also have that
\begin{eqnarray*}
\dint_0^{\pi/2}\dfrac{(\sin\theta)^{2\lz-1}}
{(\xi^2+y^2-2\xi y\cos\theta)^{\lz+1}}\,d\theta&\le&\dint_0^{\pi/2}\dfrac{\theta^{2\lz-1}}
{[(\xi-y)^2+\frac{4}{\pi^2}\xi y\theta^2]^{\lz+1}}\,d\theta\\
&\ls&\frac1{(\xi y)^\lz(\xi-y)^2}\int_0^{\fz}\frac{\bz^{2\lz-1}}{(1+\bz^2)^{\lz+1}}d\bz\\
&\sim&\frac1{x_0^{2\lz}}\frac1{(x_0-y)^2}.
\end{eqnarray*}
Combining the two cases above, we get that \eqref{cz kernel condition-3} holds.
\end{proof}

Next we recall the following estimates of the kernel $\riz(x, y)$ of the Riesz transform $\riz$, which
will be used in the sequel.  

\begin{prop}
\label{p-estimate of riesz kernel}
The Riesz kernel $\riz(x,y)$ satisfies:
\begin{itemize}
  \item [i)] There exist $K_1>2$ large enough and a positive constant $C_{K_1,\,\lz}$ such that
  for any $x,\,y\in\mathbb{R}_+$ with $y>K_1x$,
  \begin{equation}\label{lower bound of riesz kernel}
  \riz(x, y)\ge C_{K_1,\,\lz}\frac x{y^{2\lz+2}}.
  \end{equation}
  \item [ii)]There exist $K_2\in(0, 1)$ small enough and a positive constant $C_{K_2,\,\lz}$ such that
  for any $x,\,y\in\mathbb{R}_+$ with $y<K_2x$,
  \begin{equation}\label{upper bound of riesz kernel}
  \riz(x, y)\le -C_{K_2,\,\lz}\frac1{x^{2\lz+1}}.
  \end{equation}
  \item [iii)] There exist $K_3\in(1/2,2)$ such that $|K_3-1|$ small enough and a positive constant $C_{K_3,\lz}$
  such that for any $x,\,y\in\mathbb{R}_+$ with $0<|1-y/x|<|K_3-1|$,
  \begin{equation*}\label{appro estimate of riesz kernel}
 \lf |\riz(x,y)+\frac1\pi\frac1{x^\lz y^\lz}\frac1{x-y}\r|\le C_{K_3,\,\lz}\frac1{x^{2\lz+1}}\lf(\log_+\frac{\sqrt{xy}}{|x-y|}+1\r).
  \end{equation*}
\end{itemize}
\end{prop}

We point out that an earlier version of these three properties can be deduced from \cite[p.711]{bfbmt}, see also \cite[p.207]{bdt}.
See also the first version of these three estimates in \cite[p.87]{ms}.
However, to obtain our main results in Theorems \ref{t-equiv character H1} and \ref{t-commutator character BMO},
we provide the current version of these kernel estimates with the specific constants $K_1$, $K_2$ and $K_3$, and give the details of proof.

Moreover, from (iii) in Proposition \ref{p-estimate of riesz kernel} we can deduce the following inequality which will be used in the sequel.
\begin{rem}\label{r lower bound}
There exist $\tilde{K}_3\in(0, 1/2)$  small enough and a positive constant $C_{\tilde{K}_3,\lz}$
such that for any $x,\,y\in\mathbb{R}_+$ with $0<y/x-1<\tilde{K}_3$,
\begin{eqnarray*}
\riz(x,y)&\geq& {1\over \pi} {1\over x^\lambda y^\lambda}{1\over y-x} - C_{K_3,\lz}\frac1{x^{2\lz+1}}\Big( 1 + \log_+{\sqrt{xy}\over |x-y|} \Big)\\
&\geq& C_{\tilde{K}_3,\lz}{1\over x^\lambda y^\lambda}{1\over y-x}.
\end{eqnarray*}

\end{rem}

\begin{proof}[Proof of Proposition \ref{p-estimate of riesz kernel}]
For any fixed $x, y\in \mathbb{R}_+$ with $x\not=y$, write $y=kx$. Then
$k\in ((0, 1)\cup(1,\fz))$. By \eqref{riesz kernel}, we denote that $\riz(x,kx)=\frac{-2\lz}{\pi}\frac1{x^{2\lz+1}}{\rm I}$, where
\begin{equation*}
{\rm I}=\dint_0^\pi\dfrac{(1-k\cos\theta)(\sin\theta)^{2\lz-1}}
{(1+k^2-2k\cos\theta)^{\lz+1}}\,d\theta.
\end{equation*}
We estimate $\rm I$ by considering the following three cases.

Case (a) $k>2$. In this case, we write
\begin{eqnarray*}
{\rm I}&=&\int_0^{\pi}\frac\szlzm{(1+k^2-2k\cos\theta)^{\lz+1}}\,d\theta\\
&\quad&+\lf[\left(\int_0^{\pi/2}+\int_{\pi/2}^{\pi}\right)\frac{-k\cos\theta}{(1+k^2-2k\cos\theta)^{\lz+1}}\szlzmx\r]\\
&=&\int_0^{\pi}\frac\szlzm{(1+k^2-2k\cos\theta)^{\lz+1}}\,d\theta\\
&\quad&-k\int_0^{\pi/2}\lf[\frac1{(1+k^2-2k\cos\theta)^{\lz+1}}-\frac1{(1+k^2+2k\cos\theta)^{\lz+1}}\r]\cos\theta\szlzmx\\
&=:&{\rm I}_1-{\rm I}_2.
\end{eqnarray*}

As $k>2$, from \eqref{sin function integral},  we deduce that
$${\rm I}_1\le \int_0^\pi\frac\szlzm{(k-1)^{2\lz+2}}\,d\theta=\frac{\Gamma(\lz)\sqrt\pi}{\Gamma(\lz+1/2)}\frac1{(k-1)^{2\lz+2}}.$$

On the other hand, by the Mean Value Theorem, there exists $t\in(-1, 1)$ depending on $\theta$ and $\lz$, such that
\begin{eqnarray*}
{\rm I}_2&=&4k^2(\lz+1)\int_0^{\pi/2}\frac{(\cos\theta)^2}{(1+k^2-2tk\cos\theta)^{\lz+2}}\szlzmx\\
&\ge&4k^2(\lz+1)\int_0^{\pi/2}\frac{\szlzm-\szlzp}{(1+k^2+2k)^{\lz+2}}\,d\theta\\
&=&\frac{\Gamma(\lz)\sqrt\pi}{\Gamma(\lz+1/2)}\frac{k^2}{(k+1)^{2\lz+4}}\frac{\lz+1}{\lz+1/2},
\end{eqnarray*}
where the last equality follows from the second equality in \eqref{sin function integral}.
Let $a_1\in\left(1, \frac{\lz+1}{\lz+\frac{1}{2}}\right)$. Observe that there exists $K_1$ such that when $k>K_1$,
$$\frac{k^2}{(k+1)^{2\lz+4}}\frac{\lz+1}{\lz+1/2}>a_1\frac1{(k-1)^{2\lz+2}}.$$
Thus when $k>K_1$,
$$\riz(x,kx)=\frac{2\lz}{\pi}\frac1{x^{2\lz+1}}({\rm I}_2-{\rm I}_1)\ge
\frac{2\lz}{\pi}\frac1{x^{2\lz+1}}\frac{\Gamma(\lz)\sqrt\pi}{\Gamma(\lz+1/2)}\frac{a_1-1}{(k-1)^{2\lz+2}}\gs\frac x{(kx)^{2\lz+2}}.$$
This implies \eqref{lower bound of riesz kernel}.

Case (b) $k\in(0, 1)$. Similar to the argument in Case (a), we have that
$${\rm I}_1\ge\frac{\Gamma(\lz)\sqrt\pi}{\Gamma(\lz+1/2)}\frac1{(k+1)^{2\lz+2}}$$
and
$${\rm I}_2\le\frac{\Gamma(\lz)\sqrt\pi}{\Gamma(\lz+1/2)}\frac{k^2}{(k-1)^{2\lz+4}}\frac{\lz+1}{\lz+1/2}.$$
Then for some fixed $a_2\in(0, 1)$, there exists $K_2\in(0, 1)$ such that when $0<k<K_2$,
$$\frac{k^2}{(k-1)^{2\lz+4}}\frac{\lz+1}{\lz+1/2}<a_2\frac1{(k+1)^{2\lz+2}}.$$
Thus when $0<k<K_2$, there exists $C_{K_2,\,\lz}$ such that
$$\riz(x,kx)=\frac{2\lz}{\pi}\frac1{x^{2\lz+1}}({\rm I}_2-{\rm I}_1)\le
\frac{2\lz}{\pi}\frac1{x^{2\lz+1}}\frac{\Gamma(\lz)\sqrt\pi}{\Gamma(\lz+1/2)}\frac{a_2-1}{(k+1)^{2\lz+2}}\le -C_{K_2,\,\lz}\frac1{x^{2\lz+1}}.$$
This implies \eqref{upper bound of riesz kernel}.

Case (c) $k\in(1/2, 2)$. In this case, we write
$${\rm I}=\left(\int_0^{\pi/2}+\int_{\pi/2}^\pi\right)\frac{(1-k)+k(1-\cos\theta)}{(1+k^2-2k\cos\theta)^{\lz+1}}\szlzm\,d\theta=: {\rm J}_1+{\rm J}_2.$$
 To estimate ${\rm J}_1$, using Taylor's Theorem and the Mean Value Theorem, we see that for $\theta\in[0, \pi/2]$, there
exists $t_1,\,t_2,\,t_3\in (0, 1)$ such that
\begin{equation}\label{sin taylor-exp}
\szlzm=\theta^{2\lz-1}-\frac{t_1(2\lz-1)}6\theta^{2\lz+1},
\end{equation}
\begin{equation}\label{cos taylor-exp}
1-\cos\theta=\frac{\theta^2}2-\frac{t_2}{4!}\theta^4,
\end{equation}
and
\begin{eqnarray}\label{frac mean value thm}
&&\lf[(1-k)^2+k\theta^2-\frac{kt_2}{12}\theta^4\r]^{-\lz-1}\nonumber\\
&&\quad=[(1-k)^2+k\theta^2]^{-\lz-1}
+\frac{kt_2(\lz+1)}{12}\theta^4\lf[(1-k)^2+k\theta^2-\frac{kt_3}{12}\theta^4\r]^{-\lz-2}.
\end{eqnarray}
From \eqref{sin taylor-exp}, \eqref{cos taylor-exp} and \eqref{frac mean value thm}, it follows that
\begin{eqnarray*}
{\rm J}_1&=&\int_0^{\pi/2}\frac{[(1-k)+(\frac k2\theta^2-\frac{t_2k}{4!}\theta^4)](\theta^{2\lz-1}-\frac{t_1(2\lz-1)}6\theta^{2\lz+1})}
{[(1-k)^2+k\theta^2-\frac{kt_2}{12}\theta^4]^{\lz+1}}\,d\theta\\
&=&\int_0^{\pi/2}\frac{(1-k)\theta^{2\lz-1}}{[(1-k)^2+k\theta^2]^{\lz+1}}\,d\theta\\
&\quad&+\int_0^{\pi/2}\frac{(\frac{k}{2}-\frac{2\lz-1}{6}(1-k)t_1)\theta^{2\lz+1}
-(\frac{2\lz-1}{12}t_1k+\frac{t_2k}{24})\theta^{2\lz+3}+\frac{2\lz-1}{144}t_1t_2k\theta^{2\lz+5}}{[(1-k)^2+k\theta^2]^{\lz+1}}\,d\theta\\
&\quad&+\frac{k}{12}(\lz+1)\int_0^{\pi/2}t_2\lf\{(1-k)\theta^{2\lz+3}+\lf[\frac{k}{2}-\frac{2\lz-1}{6}(1-k)t_1\r]\theta^{2\lz+5}\r.\\
&\quad&\lf.-\lf(\frac{t_1k(2\lz-1)}{12}+\frac{t_2k}{24}\r)\theta^{2\lz+7}
+\frac{(2\lz-1)t_1t_2}{144}k\theta^{2\lz+9}\r\}\\
&\quad&\times\frac1{[(1-k)^2+k\theta^2-\frac{kt_3}{12}\theta^4]^{\lz+2}}\,d\theta\\
&=:&{\rm J}_{11}+{\rm J}_{12}+{\rm J}_{13}.
\end{eqnarray*}

Observe that
\begin{equation*}\label{beta fun}
\inzf\frac{\beta^{2\lz-1}}{(1+\beta^2)^{\lz+1}}\,d\bz=\frac12B(\lz,1)=\frac1{2\lz},
\end{equation*}
where $B(p,q)$ is the Beta function. Then we have that
\begin{eqnarray*}
{\rm J}_{11}&=&\frac{1-k}{|1-k|^{2\lz+2}}\int_0^{\pi/2}\frac{\theta^{2\lz-1}}{[1+(\frac{\sqrt k}{|1-k|}\theta)^2]^{\lz+1}}\,d\theta\\
&=&\frac1{k^\lz}\frac1{1-k}\lf[\frac1{2\lz}-\int_{\frac\pi2\frac{\sqrt k}{|1-k|}}^\fz\frac{\bz^{2\lz-1}}{(1+\bz^2)^{\lz+1}}\,d\bz\r].
\end{eqnarray*}
By this and the fact that
$$0<\int_{\frac\pi2\frac{\sqrt k}{|1-k|}}^\fz\frac{\bz^{2\lz-1}}{(1+\bz^2)^{\lz+1}}\,d\bz<\frac{2}{\pi^2}\frac{(k-1)^2}{k},$$
we see that ${\rm J}_{11}-\frac1{2\lz}\frac1{k^\lz}\frac1{1-k}\to0$, $k\to1$.

Similarly, we have that
\begin{eqnarray*}
|{\rm J}_{12}|&\ls& \int_0^{\pi/2}\frac{\theta^{2\lz+1}+\theta^{2\lz+3}
+\theta^{2\lz+5}}{[(1-k)^2+k\theta^2]^{\lz+1}}\,d\theta\\
&\ls& \frac1{(1-k)^{2\lz+2}}\int_0^{\pi/2}\frac{\theta^{2\lz+1}}{[1+(\frac{\sqrt k}{|k-1|}\theta)^2]^{\lz+1}}\,d\theta
+\int_0^{\pi/2}\frac{\theta^{2\lz+3}
+\theta^{2\lz+5}}{[(1-k)^2+k\theta^2]^{\lz+1}}\,d\theta\\
&\ls&\int_0^{\frac{ \pi}{2}\frac{\sqrt k}{|k-1|}}\frac{\bz^{2\lz+1}}{(1+\bz^2)^{\lz+1}}\,d\bz
+\int_0^{\pi/2}(\theta+\theta^3)\,d\theta\\
&\ls& \int_0^1\bz^{2\lz+1}\,d\bz+\int_1^{\frac{\sqrt k}{|k-1|}}\bz^{-1}\,d\bz+1\\
&\ls&\log_+\frac{\sqrt k}{|k-1|}+1,
\end{eqnarray*}
and
\begin{eqnarray*}
|{\rm J}_{13}|&\ls&\int_0^{\pi/2}\frac{\theta^{2\lz+3}+\theta^{2\lz+5}
+\theta^{2\lz+7}+\theta^{2\lz+9}}{[(1-k)^2+k\theta^2/4]^{\lz+2}}\,d\theta\\
&\ls& \frac1{(1-k)^{2\lz+4}}\int_0^{\pi/2}\frac{\theta^{2\lz+3}}{[1+(\frac{\sqrt k}{2|k-1|}\theta)^2]^{\lz+2}}\,d\theta+1
\ls\log_+\frac{\sqrt k}{|k-1|}+1.
\end{eqnarray*}

For the estimate of ${\rm J}_2$, since  $\cos\theta\le0$ when $\theta\in[\pi/2,\pi]$, it is easy to see that
${\rm J}_2\ls1$. Combining the estimates of ${\rm J}_{11}$, ${\rm J}_{12}$, ${\rm J}_{13}$ and ${\rm J}_2$,
we finish the proof of Proposition \ref{p-estimate of riesz kernel}.
\end{proof}


We now recall the Hardy and BMO spaces associated with $S_\lz$.
Betancor et al. \cite{bdt}  introduced an equivalent definition of the Hardy spaces  $H^1_{S_\lambda}(\mathbb{R}_+,dx)$  associated with $S_\lz$ via the maximal function via Poisson semigroups, i.e.
\begin{equation*} \label{Hardy m for bessel 2}
H^1_{S_\lambda}(\mathbb{R}_+,\,dx):=\Big\{ f\in L^1(\mathbb{R}_+,\,dx): \sup_{t>0}|e^{-t\sqrt{S_\lambda}}(f)|\in L^1(\mathbb{R}_+,\,dx) \Big\}
\end{equation*}
with norm $$\|f\|_{H^1_{S_\lambda}(\mathbb{R}_+,\,dx)}= \|f\|_{L^1(\mathbb{R}_+,\,dx)} + \|\sup_{t>0}|e^{-t\sqrt{S_\lambda}}(f)|\|_{L^1(\mathbb{R}_+,\,dx)}.$$
Moreover, they proved that $H^1_{S_\lambda}(\mathbb{R}_+,dx)$ is equivalent to the Hardy space $H_o^1(\mathbb{R}_+,dx)$ (see Theorem 3.1 and Proposition 3.9 in \cite{bdt}), where $H_o^1(\mathbb{R}_+,dx)$ is defined in \cite{CKS} as follows
$$H_o^1(\mathbb{R}_+,dx)= \{ f\in L^1(\mathbb{R}_+,dx):  f_o \in H^1(\mathbb{R})\}$$
with the norm defined by $\|f\|_{H_o^1(\mathbb{R}_+,\,dx)}:=\|f_o\|_{H^1(\mathbb{R})}$.  Here $f_o(x):=f(x) $ if $x\in \mathbb{R}_+$, $f_o(x):=-f(-x) $ if $x\in \mathbb{R}_-$, which is also called the odd extension of $f$ on $\mathbb{R}_+$. We point out that the atoms in $H_o^1(\mathbb{R}_+,dx)$ may not have cancellation property.
It follows from Chang et al. \cite{CKS} that the dual space of $H_o^1(\mathbb{R}_+,dx)$ is BMO$_z(\mathbb{R},dx)$. We further point out that as proved in \cite[Proposition 3.1]{DDSY}, this BMO$_z(\mathbb{R},dx)$ is equivalent to BMO$_o(\mathbb{R}_+,dx)$, which is defined as
\begin{equation}\label{BMOo}
{\rm BMO}_o(\mathbb{R}_+,dx) = \{f\in L^1_{\rm loc}(\mathbb{R}_+,\,dx): f_o \in {\rm BMO}(\mathbb{R})\}.
\end{equation}
where ${\rm BMO}(\mathbb{R})$ is the standard BMO space on $\mathbb{R}$ introduced by John--Nirenberg. Thus, we have that
\begin{equation}\label{BMOoo}
{\rm BMO}_{S_\lambda}(\mathbb{R}_+,dx) = {\rm BMO}_o(\mathbb{R}_+,dx).
\end{equation}

We finally note that the Riesz transform $R_{S_\lambda}$  related to Bessel operator $S_\lambda$ (defined as in \eqref{bessel 2})
is bounded on $L^2(\mathbb{R}_+,dx)$, and the kernel $R_{S_\lambda}(x,y)$ of $R_{S_\lambda}$
satisfies the following size and regularity properties, as proved in \cite[Proposition 4.1]{bfbmt}.
\begin{prop}[\cite{bfbmt}]\label{p upper}
There exists $C>0$ such that for every $x,y\in \mathbb{R}_+$ with $x\not=y$,
\begin{eqnarray*}
&& (i)\  |R_{S_\lambda}(x,y)|\leq {C\over |x-y|};\\
&& (ii)\  \Big|{\partial\over \partial  x} R_{S_\lambda}(x,y)\Big| + \Big|{\partial\over \partial y} R_{S_\lambda}(x,y)\Big| \leq {C\over |x-y|^2}.\\
\end{eqnarray*}
\end{prop}


\section{Hardy space factorization and BMO space characterization in the setting of~$\Delta_{\lambda}$}
\label{s:MainResult}

In this section we provide the details of the proof of Theorems \ref{t-equiv character H1} and \ref{t-commutator character BMO}, in the following
structure: we first provide the proof of (1) in Theorem \ref{t-commutator character BMO}, which plays the key role for
the proof of Theorem \ref{t-equiv character H1}. Then the proof of (2) in Theorem \ref{t-commutator character BMO} follows from
Theorem \ref{t-equiv character H1}.

\begin{proof}[Proof of (1) of Theorem \ref{t-commutator character BMO}]
We first prove the upper bound, i.e.,
for $b\in \bmoz$ and $p\in(1, \fz)$, there exists a positive constant $C$
such that for any $f\in \lpz$,
\begin{equation}\label{upper bd}
\lf\|[b, \riz]f\r\|_\lpz\le C\|f\|_\lpz.
\end{equation}

Note that \eqref{upper bd} follows from
 \cite[Theorem 2.5]{bc} since
the Riesz transform $\riz$ is a Calder\'on--Zygmund operator as indicated in Proposition \ref{t:RieszCZ}.
\end{proof}

We now prove Theorem \ref{t-equiv character H1} based on (1) of Theorem \ref{t-commutator character BMO}. To begin with, we now provide an auxiliary lemma for the Hardy space $\hoz$ associated with
the Bessel operator, which plays an important role in the
proof of our main result.

\begin{lem}\label{l-atomic estimate}
Let $f$ be a function satisfying the following estimates:
\begin{itemize}
  \item [i)]   $\inzf f(x)\xtz=0$;
  \item [ ii)] there exist intervals $I(x_1, r)$ and $I(x_2, r)$ for some $x_1, x_2, r\in\mathbb{R}_+$ and positive constants
  $D_1,\,D_2$ such that
  $$|f(x)|\le D_1\chi_{I(x_1, \,r)}(x)+D_2\chi_{I(x_2,\,r)}(x);$$
  \item [iii)]$|x_1-x_2|\ge4r$.
\end{itemize}
Then there exists a positive constant $C$ independent of $x_1,x_2,r,D_1,D_2$, such that
$$\|f\|_\hoz\le C\log_2\frac{|x_1-x_2|}r\left[D_1m_\lz(x_1,r)+D_2m_\lz(x_2,r)\right].$$

\end{lem}

\begin{proof}
Assume that $f:=f_1+f_2$, where $\supp f_i\subset I(x_i,r) $ for $i=1,\,2$.
We will show that $f$ has the following $(1, \fz)_{\Delta_\lz}$-atomic decomposition
\begin{equation}\label{atomic decom}
f=\sum_{i=1}^2\sum_{j=1}^{J_0+1}\az^j_ia^j_i,
\end{equation}
where $J_0$ is the smallest integer larger than $\log_2\frac{|x_1-x_2|}r$, for each $j$, $a^j_i$ is
 a $(1, \fz)_{\Delta_\lz}$-atom and $\az^j_i$ a real number satisfying that
\begin{equation}\label{coefficient bdd}
\lf|\az^j_i\r|\ls D_im_\lz(I(x_i, r)).
\end{equation}

To this end, we write
$$f=\sum_{i=1}^2\lf[f_i-\wz\az^1_i\chi_{I(x_i,\,2r)}\r]+\sum_{i=1}^2\wz\az^1_i\chi_{I(x_i,\,2r)}
=:f^1_1+f^1_2+\sum_{i=1}^2\wz\az^1_i\chi_{I(x_i,\,2r)},$$
where
$$\wz\az^1_i:=\frac1{m_\lz(I(x_i,\,2r))}\int_{I(x_i,\,r)}f_i(x)\xtz.$$
Let
$$\az^1_i:=\lf\|f^1_i\r\|_\linz m_\lz(I(x_i,\,2r))$$
and $a^1_i:=f^1_i/\az^1_i$. Then
we see that $a^1_i$ is a $(1, \fz)_{\Delta_\lz}$-atom supported on $I(x_i, 2r)$ and
$\az^1_i$ satisfies \eqref{coefficient bdd}.

For $i=1,\,2$, we further write
\begin{eqnarray*}
\wz\az^1_i\chi_{I(x_i,\,2r)}
&=&\wz\az^1_i\chi_{I(x_i,\,2r)}-\wz\az^2_i\chi_{I(x_i,\,4r)}+\wz\az^2_i\chi_{I(x_i,\,4r)}
=:f^2_i+\wz\az^2_i\chi_{I(x_i,\,4r)},
\end{eqnarray*}
where $$\wz\az^2_i:=\frac1{m_\lz(I(x_i,\,4r))}\int_{I(x_i,\,r)}f_i(x)\xtz.$$
Let
$$\az^2_i:=\lf\|f^2_i\r\|_\linz m_\lz(I(x_i,\,4r))$$
 and $a^2_i:=f^2_i/\az^2_i$. Then
we see that $a^2_i$ is a $(1, \fz)_{\Delta_\lz}$-atom supported on $I(x_i, 4r)$ and
$$\lf|\az^2_i\r|\le\lf|\wz\az^1_i\r|m_\lz(I(x_i,\,4r))\le\frac{m_\lz(I(x_i,\,4r))}{m_\lz(I(x_i,\,2r))}
\|f_i\|_\linz m_\lz(I(x_i,\,r))\ls D_i m_\lz(I(x_i,\,r)).$$

Continuing in this fashion we see that for $j\in\{1,\,2,\,\ldots,J_0\}$,
$$f=\sum_{i=1}^2\lf[\sum_{j=1}^{J_0}f^j_i\r]+\sum_{i=1}^2\wz\az^{J_0}_i\chi_{I(x_i,\,2^{J_0}r)}
=\sum_{i=1}^2\lf[\sum_{j=1}^{J_0}\az^j_ia^j_i\r]+\sum_{i=1}^2\wz\az^{J_0}_i\chi_{I(x_i,\,2^{J_0}r)},$$
where for $j\in\{2,\, 3,\,\ldots,\,J_0\}$,
$$\wz\az^j_i:=\frac1{m_\lz(I(x_i,\,2^jr))}\int_{I(x_i,\,r)}f_i(x)\xtz,$$
$$f^j_i:=\wz \az^{j-1}_i\chi_{I(x_i,\,2^{j-1}r)}-\wz \az^j_i\chi_{I(x_i,\,2^jr)},$$
$$\az^j_i:=\lf\|f^j_i\r\|_\linz m_\lz(I(x_i, 2^jr))\,\,{\rm and}\,a^j_i:=f^j_i/\az^j_i.$$
Moreover, for each $i$ and $j$, $a^j_i$ is a a $(1, \fz)_{\Delta_\lz}$-atom and $|\az^j_i|\ls D_im_\lz(I(x_i, r))$.

For $\sum_{i=1}^2\wz\az^{J_0}_i\chi_{I(x_i,\,2^{J_0}r)}$, we set
\begin{eqnarray*}
\wz\az^{J_0}&:=&\frac1{m_\lz(I(\frac{x_1+x_2}2,\,2^{J_0+1}r))}\int_{I(x_1,\,r)}f_1(x)\xtz\\
&=&-\frac1{m_\lz(I(\frac{x_1+x_2}2,\,2^{J_0+1}r))}\int_{I(x_2,\,r)}f_2(x)\xtz.
\end{eqnarray*}
Then
\begin{eqnarray*}
&&\sum_{i=1}^2\wz\az^{J_0}_i\chi_{I(x_i,\,2^{J_0}r)}\\
&&\quad=\lf[\wz\az^{J_0}_1\chi_{I(x_1,\,2^{J_0}r)}-\wz\az^{J_0}\chi_{I(\frac{x_1+x_2}2,\,2^{J_0+1}r)}\r]
+\lf[\wz\az^{J_0}\chi_{I(\frac{x_1+x_2}2,\,2^{J_0+1}r)}+\wz\az^{J_0}_2\chi_{I(x_2,\,2^{J_0}r)}\r]\\
&&\quad=:\sum_{i=1}^2f^{J_0+1}_i.
\end{eqnarray*}
For $i=1,\,2$, let
$$\az^{J_0+1}_i:=\lf\|f^{J_0+1}_i\r\|_\linz m_\lz\lf(I\lf(\frac{x_1+x_2}2,\,2^{J_0+1}r\r)\r)\,\, {\rm and}\,\,
a^{J_0+1}_i:=f^{J_0+1}_i/\az^{J_0+1}_i.$$
Then we see that $a^{J_0+1}_i$ is a $(1, \fz)_{\Delta_\lz}$-atom and $\az^{J_0+1}_i$ satisfies \eqref{coefficient bdd}.
Thus, we have \eqref{atomic decom} holds, which implies that $f\in \hoz$ and
\begin{equation*}
\|f\|_\hoz\le \sum_{i=1}^2\sum_{j=1}^{J_0+1}\lf|\az_i^j\r|\ls \log\frac{|x_1-x_2|}r\sum_{i=1}^2D_im_\lz(I(x_i,r)).
\end{equation*}
This finishes the proof of Lemma \ref{l-atomic estimate}.
\end{proof}

\begin{rem}\label{r Hardy atom}
From the proof of the Lemma \ref{l-atomic estimate}, we see that this result holds for general Hardy space $H^1(X,d,\mu)$ on
the spaces of homogeneous type $(X,d,\mu)$ in the sense of Coifman and Weiss \cite{cw77}.
\end{rem}


Next we provide the following estimate of the bilinear operator $\Pi$, which is defined in \eqref{def of pi}.
\begin{prop}\label{t-H1 estimate of pi}
Let $p\in (1, \fz)$. There exists a positive constant $C$ such that for any $g\in\lpz$ and $h\in\lppz$,
$$\|\Pi(g,h)\|_\hoz\le C\|g\|_\lpz\|h\|_\lppz.$$
\end{prop}

\begin{proof}
Since $\riz$ and $\wriz$ are both bounded on $\lrz$ for any $r\in(1, \fz)$, we see that $\Pi(g,h)\in\loz$ for
any $g\in\lpz$ and $h\in\lppz$ and
$$\inzf\Pi(g, h)(x)x^{2\lz}\,dx=0.$$
Moreover, from (1) of Theorem \ref{t-commutator character BMO}, it follows that
for every $f\in \bmoz$, $g\in\lpz$ and $h\in\lppz$,
\begin{eqnarray*}
\lf|\inzf f(x)\Pi(g,h)(x)\xtz\r|&=&\lf|\inzf g(x)[f, \riz] h(x)\xtz\r|\\
&\ls&\|h\|_\lppz\|g\|_\lpz\|f\|_\bmoz.
\end{eqnarray*}
Therefore, the proof of Proposition \ref{t-H1 estimate of pi} is completed.
\end{proof}

%

The following proposition will lead to an iterative argument to prove the lower bound appearing in Theorem \ref{t-equiv character H1}.

\begin{prop}\label{t-H1 appro}
Let $p\in(1, \fz)$. For every $\ez>0$,
there exist positive constants $M$ and $C$ such that for every $(1,\fz)_{\Delta_\lz}$-atom $a$, there exist $g\in\lpz$ and $h\in\lppz$ satisfying that
\begin{equation*}\label{H1 appro}
\|a-\Pi(g,h)\|_\hoz<\ez
\end{equation*}
and $\|g\|_\lpz\|h\|_\lppz\le CM^{\frac{2\lz}p+1}.$
\end{prop}

\begin{proof}
Assume that $a$ is a $(1,\fz)_{\Delta_\lz}$-atom with $\supp a \subset I(x_0, r)$. Observe that if $r>x_0$, then
$I(x_0, r)=(x_0-r,x_0+r)\cap \mathbb{R}_+=I(\frac{x_0+r}2, \frac{x_0+r}2)$. Therefore, without loss of generality, we may assume
that $r\le x_0$. Let $K_2$ and $\tilde{K}_3$ be the constants appeared in  (ii) of Proposition \ref{p-estimate of riesz kernel} and Remark \ref{r lower bound} respectively, and
$K_0>\max\{{1\over K_2},\,\frac1{\tilde{K}_3}\}+1$ large enough. For any $\ez>0$, let $M$ be a positive constant large enough such that $M\ge100K_0$
and $\frac{\log_2 M}{M}<\ez$.

We now consider the following two cases.

Case (a): $x_0\le 2Mr$. In this case, let $y_0:=x_0+2MK_0r$. Then
$$(1+K_0)x_0\le y_0\le (1+2MK_0)x_0.$$
Define
$$g(x):=\chi_{I(y_0,\,r)}(x)\,\,{\rm and}\,\,h(x):=-\frac{a(x)}{\wriz g(x_0)}.$$

\noindent By the fact that $y/x_0>K_2^{-1}$ for any $y\in I(y_0,r)$ and Proposition \ref{p-estimate of riesz kernel} \textit{ii)}, we see that
\begin{equation}\label{lower bound of riesz}
\lf|\wriz g(x_0)\r|=\Big|\int_{y_0-r}^{y_0+r}\riz(y, x_0)y^{2\lz}dy\Big|\gs\int_{y_0-r}^{y_0+r}\frac1ydy\sim\frac ry_0\sim\frac1M.
\end{equation}

\noindent Moreover, from the definitions of $g$ and $h$,  it follows that
\begin{eqnarray*}
\|g\|_\lpz\|h\|_\lppz&\le& \frac1{|\wriz g(x_0)|}\lf[\myzr\r]^{\frac1p}\lf[\mxzr\r]^{-\frac1p}\\
&\ls&M\lf(y_0^{2\lz}r\r)^{1/p}\lf(x^{2\lz}_0r\r)^{-1/p}\ls M^{\frac{2\lz}p+1}.
\end{eqnarray*}

By the definition of the operator $\Pi$ as in \eqref{def of pi}, we write
\begin{eqnarray*}
a(x)-\Pi(g,h)(x)&=&a(x)\frac{\wriz g(x_0)-\wriz g(x)}{\wriz g(x_0)}-g(x)\riz h(x)=: W_1(x)+W_2(x).
\end{eqnarray*}
Then it is obvious that $\supp W_1\subset I(x_0, r)$ and $\supp W_2\subset I(y_0, r)$. Moreover, let
$$C_1:=\frac1\mxzr\frac\myzr\mxzxy\,\, {\rm and}\,\,C_2:=\frac1\mxzxy.$$
From the cancellation property $\inzf a(y)y^{2\lz}\,dy=0$, the H\"older's regularity of the Riesz kernel $\riz(x, y)$ in \eqref{cz kernel condition-2}, and the fact that $|y-x|\sim|x_0-y_0|$ for $y\in I(x_0, r)$ and $x\in I(y_0, r)$, we have
\begin{eqnarray*}
|W_2(x)|&=&\chi_{I(y_0,\, r)}(x)\lf|\riz h(x)\r|\\
&\ls&M\chi_{I(y_0,\, r)}(x)\lf|\int_{I(x_0,\,r)}\lf[\riz(x,y)-\riz(x,x_0)\r]a(y)y^{2\lz}\,dy\r|\\
&\ls&M\chi_{I(y_0,\, r)}(x)\frac{r}{|x_0-y_0|}\frac1{m_\lz(I(y_0,\, |x_0-y_0|))}\\
&\ls& C_2\chi_{I(y_0,\, r)}(x).
\end{eqnarray*}
On the other hand, using \eqref{cz kernel condition-2} and $|y-x_0|\sim|x_0-y_0|$ for $y\in I(y_0, r)$,
\begin{eqnarray*}
|W_1(x)|&\le&M\chi_{I(x_0,\, r)}(x)\|a\|_\linz\int_{I(y_0,\,r)}\frac{r}{|x_0-y|}\frac{1}{m_\lz(I(x_0,\,|x_0-y|))}y^{2\lz}\,dy\\
&\ls&M\chi_{I(x_0,\, r)}(x)\frac1\mxzr\frac r{|x_0-y_0|}\frac\myzr\mxzxy\ls C_1\chi_{I(x_0,\, r)}(x).
\end{eqnarray*}

Moreover, note that
$$\inzf[a(x)-\Pi(g, h)(x)]\,\xtz=0.$$

Hence, the function $f(x)= a(x)-\Pi(g, h)(x)$ satisfies all conditions in Lemma \ref{l-atomic estimate}.
Now from Lemma \ref{l-atomic estimate}, we have that
\begin{eqnarray}\label{appro estimate of atom}
\|a-\Pi(g,h)\|_\hoz&\ls&\log_2\Big(\frac{|x_0-y_0|}r\Big)\lf[C_1\mxzr+C_2\myzr\r]\nonumber\\
&\ls&\log_2\Big(\frac{|x_0-y_0|}r\Big)\,\frac r{|x_0-y_0|}\nonumber\\
&\ls& {\log_2 M \over M}<\ez.
\end{eqnarray}

Case (b): $x_0>2Mr$. In this case, let $y_0:=x_0-Mr/K_0$. Then ${2K_0-1\over 2K_0}x_0<y_0<x_0$. Let $g$ and $h$ be as in Case (a).
For every $y\in I(y_0, r)$, from the facts that $K_0>\max\{{1\over K_2},\,\frac1{\tilde{K}_3}\}+1$ and $M\ge100K_0$, we have
$$0<{x_0\over y} -1<\tilde{K}_3.$$
By Remark \ref{r lower bound}
and $y\sim y_0\sim x_0$ for any $y\in I(y_0, r)$,
we conclude that
\begin{eqnarray}\label{lower bound of riesz-2}
\lf|\wriz g(x_0)\r|\gs \Big|\int_{y_0-r}^{y_0+r}\frac1{x_0^\lz y_0^\lz}\frac1{x_0-y}y^{2\lz}\,dy\Big|
\sim\int_{y_0-r}^{y_0+r}\frac1{x_0-y_0}\,dy\sim\frac1M.
\end{eqnarray}
Moreover,
$$\|g\|_\lpz\|h\|_\lppz\ls M\lf[\frac{\myzr}{\mxzr}\r]^{1/p}\sim M.$$

Let $W_1$,$W_2$, $C_1$ and $C_2$ be as in Case (i). Then similarly, we have that
\begin{eqnarray*}
|W_2(x)|&\ls&M\chi_{I(y_0,\, r)}(x)\lf|\int_{I(x_0,\,r)}[\riz(x,y)-\riz(x, x_0)]a(y)y^{2\lz}\,dy\r|\\
&\ls&M\chi_{I(y_0,\, r)}(x)\frac{r}{|x_0-y_0|}\frac1{m_\lz(I(y_0,\, |x_0-y_0|))}\sim C_2\chi_{I(y_0,\, r)}(x),
\end{eqnarray*}
and
\begin{eqnarray*}
|W_1(x)|&\le&M\chi_{I(x_0,\, r)}(x)\|a\|_\linz\int_{I(y_0,\,r)}\frac{r}{|x_0-y|}\frac{1}{m_\lz(I(x_0,\,|x_0-y|))}y^{2\lz}\,dy\\
&\ls&M\chi_{I(x_0,\, r)}(x)\frac1\mxzr\frac r{|x_0-y_0|}\frac\myzr\mxzxy\ls C_1\chi_{I(x_0,\, r)}(x).
\end{eqnarray*}
Then \eqref{appro estimate of atom} follows from Lemma \ref{l-atomic estimate} in this case, which
together with Case (a) completes the proof of Proposition \ref{t-H1 appro}.
\end{proof}

We now use the above proposition in an iterative fashion to deduce the first main result Theorem \ref{t-equiv character H1}.


\begin{proof}[Proof of Theorem \ref{t-equiv character H1}]
By Proposition \ref{t-H1 estimate of pi}, we have that for any $g\in\lpz$ and $h\in\lppz$,
$$\|\Pi(g, h)\|_\hoz\ls \|g\|_\lpz\|h\|_\lppz.$$
From this, for any $f\in\hoz$  having the representation \eqref{represent of H1} with
$$\sum_{k=1}^\fz\sum_{j=1}^\fz\lf|\az^k_j\r|\lf\|g^k_j\r\|_\lpz\lf\|h^k_j\r\|_\lppz<\fz,$$
 it follows that
\begin{eqnarray*}
\|f\|_\hoz&\ls&\inf\lf\{\sum_{k=1}^\fz\sum_{j=1}^\fz\lf|\az^k_j\r|\lf\|g^k_j\r\|_\lpz\lf\|h^k_j\r\|_\lppz:\r.\\
&\quad&\lf.\quad \quad f=\sum_{k=1}^\fz\sum_{j=1}^\fz \az^k_j\,\Pi\lf(g^k_j,h^k_j\r)\r\}.
\end{eqnarray*}

To see the converse, let $f\in\hoz$. We will show that $f$ has a representation as in \eqref{represent of H1} with
\begin{equation}\label{lower bound of H1 respent}
\inf\lf\{\sum_{k=1}^\fz\sum_{j=1}^\fz\lf|\az^k_j\r|\lf\|g^k_j\r\|_\lpz\lf\|h^k_j\r\|_\lppz\r\}\ls \|f\|_\hoz.
\end{equation}
To this end, assume that $f$ has the following atomic representation
 $f=\sum_{j=1}^\fz\az^1_ja^1_j$ with $\sum_{j=1}^\fz|\az^1_j|\le C_3\|f\|_\hoz$
 for certain constant $C_3\in(1, \fz)$.
We show that for any $\epsilon\in(0, 1/C_3)$ and
any $K\in\nn$, $f$ has the following representation
\begin{equation}\label{itration}
f=\sum_{k=1}^K\sum_{j=1}^\fz\az^k_j\,\Pi\lf(g^k_j, h^k_j\r)+E_K,
\end{equation}
where, $M$ is as in Proposition \ref{t-H1 appro}, $g^k_j\in\lpz$, $h^k_j\in\lppz$ for each $k$ and $j$,
$\{\az^k_j\}_{j}\in \ell^1$ for each $k$ and $E_K\in \hoz$ satisfying that
\begin{equation}\label{itration-1}
\lf\|g^k_j\r\|_\lpz\lf\|h^k_j\r\|_\lppz\ls M^{\frac{2\lz}p+1},
\end{equation}
\begin{equation}\label{itration-2}
\sum_{j=1}^\fz\lf|\az^k_j\r|\le  \ez^{k-1}C_3^k\|f\|_\hoz
\end{equation}
and
\begin{equation}\label{itration-3}
\|E_K\|_\hoz\le (\ez C_3)^K\|f\|_\hoz.
\end{equation}

 In fact, for given $\ez$ and $a_j$, by Proposition \ref{t-H1 appro}, there exist
 $g^1_j\in\lpz$ and $h^1_j\in\lppz$ with
 $$\lf\|g^1_j\r\|_\lpz\lf\|h^1_j\r\|_\lppz\ls M^{\frac{2\lz}p+1}$$
  and
 $$\lf\|a^1_j-\Pi\lf(g^1_j,h^1_j\r)\r\|_\hoz<\ez.$$
Now we write
$$f=\sum_{j=1}^\fz\az^1_ja^1_j=\sum_{j=1}^\fz\az^1_j\Pi\lf(g^1_j,h^1_j\r)+\sum_{j=1}^\fz\az^1_j\lf[a^1_j-\Pi\lf(g^1_j,h^1_j\r)\r]
=:M_1+E_1.$$
Observe that
\begin{equation*}
\|E_1\|_\hoz\le \sum_{j=1}^\fz\lf|\az^1_j\r|\lf\|a^1_j-\Pi\lf(g^1_j,h^1_j\r)\r\|_\hoz\le \ez C_3\|f\|_\hoz.
\end{equation*}

Since $E_1\in\hoz$, for the given $C_3$, there exist a sequence of atoms $\{a^2_j\}_j$ and numbers $\{\az^2_j\}_j$
such that $E_2=\sum_{j=1}^\fz\az^2_ja^2_j$ and
\begin{equation*}
\sum_{j=1}^\fz\lf|\az^2_j\r|\le C_3\|E_1\|_\hoz\le \ez C_3^2\|f\|_\hoz.
\end{equation*}

Another application of Proposition \ref{t-H1 appro} implies that there exist functions $g^2_j\in\lpz$ and $h^2_j\in\lppz$ with
\begin{equation*}
\lf\|g^2_j\r\|_\lpz\lf\|h^2_j\r\|_\lppz\ls M^{\frac{2\lz}p+1}\,\,{\rm and}\,\, \lf\|a^2_j-\Pi\lf(g^2_j, h^2_j\r)\r\|_\hoz<\ez.
\end{equation*}
Thus, we have
$$E_1=\sum_{j=1}^\fz\az^2_ja^2_j=\sum_{j=1}^\fz\az^2_j\Pi\lf(g^2_j,h^2_j\r)+\sum_{j=1}^\fz\az^2_j\lf[a^2_j-\Pi\lf(g^2_j,h^2_j\r)\r]
=:M_2+E_2.$$
Moreover,
\begin{eqnarray*}
\|E_2\|_\hoz&\le& \sum_{j=1}^\fz\lf|\az^2_j\r|\lf\|a^2_j-\Pi\lf(g^2_j,h^2_j\r)\r\|_\hoz
\le\ez\sum_{j=1}^\fz\lf|\az^2_j\r|\le (\ez C_3)^2\|f\|_\hoz.
\end{eqnarray*}
Now we conclude that
\begin{eqnarray*}
f=\sum_{j=1}^\fz\az^1_ja^1_j=\sum_{k=1}^2\sum_{j=1}^\fz\az^k_j\Pi\lf(g^k_j, h^k_j\r)+E_2,
\end{eqnarray*}

Continuing in this way, we deduce that for any $K\in\nn$, $f$ has the representation \eqref{itration} satisfying
\eqref{itration-1}, \eqref{itration-2} and \eqref{itration-3}. Thus letting $K\to\fz$, we
see that \eqref{represent of H1} holds. Moreover, since $\ez C_3<1$, we have that
$$\sum_{k=1}^\fz\sum_{j=1}^\fz \lf|\az^k_j\r|\le \sum_{k=1}^\fz\ez^{-1}(\ez C_3)^k\|f\|_\hoz\ls \|f\|_\hoz,$$
which implies \eqref{lower bound of H1 respent} and hence, completes the proof of Theorem \ref{t-equiv character H1}.
\end{proof}

Next we turn to the proof of (2) of Theorem \ref{t-commutator character BMO}. 

\begin{proof}[Proof of (2) of Theorem \ref{t-commutator character BMO}]

Assume that $[b,\riz]$ is bounded on $\lppz$ for a given $p'\in (1, \fz)$ and
$$f\in(\hoz\cap L^\fz_c(\rr_+,dm_\lz)),$$
where $L^\fz_c(\rr_+, dm_\lz)$ is the subspace of $\linz$ consisting
of functions with compact supports in $\rr_+$. From Theorem \ref{t-equiv character H1},
we deduce that
\begin{eqnarray*}
\langle b, f\rangle&=&\sum_{k=1}^\fz\sum_{j=1}^\fz \az^k_j\lf\langle b, \Pi\lf(g^k_j,h^k_j\r)\r\rangle
=\sum_{k=1}^\fz\sum_{j=1}^\fz \az^k_j\lf\langle g^k_j,[b,\riz]h^k_j\r\rangle.
\end{eqnarray*}
This implies that
\begin{eqnarray*}
\lf|\langle b, f\rangle\r|
&&\le\sum_{k=1}^\fz\sum_{j=1}^\fz \lf|\az^k_j\r|\lf\|g^k_j\r\|_\lpz\lf\|[b,\riz]h^k_j\r\|_\lppz\\
&&\le\lf\|[b, \riz]\r\|_{\lppz\to\lppz}\sum_{k=1}^\fz\sum_{j=1}^\fz \lf|\az^k_j\r|\lf\|g^k_j\r\|_\lpz\lf\|h^k_j\r\|_\lppz\\
& &\ls\lf\|[b, \riz]\r\|_{\lppz\to\lppz}\|f\|_\hoz.
\end{eqnarray*}
Then by the fact that $\hoz\cap L^\fz_c(\rr_+,dm_\lz)$ is dense in $\hoz$ and
the duality between $\hoz$ and $\bmoz$, we finish the proof of Theorem \ref{t-commutator character BMO}.
\end{proof}

\medskip

\section{Proof of Theorems \ref{t-commutator not character BMO} and Corollary \ref{c weak factorization}}
\label{s:MainResult 2}

In this section, we give the proofs of Theorem \ref{t-commutator not character BMO} and Corollary \ref{c weak factorization}.

\medskip

\begin{proof}[\bf Proof of Theorem \ref{t-commutator not character BMO}]

From (i) and (ii) in Proposition \ref{p upper}, we get that  $R_{S_\lambda}$ falls into the scope of classical Calder\'on--Zygmund operators (see for example \cite{s93}). Hence, by the
result of Coifman et al. \cite{crw}, we have that:  For $1<p<\infty,$  if $b$ lies in the classical BMO space BMO$(\mathbb{R}_+,dx)$ in the sense of John--Nirenberg, then the commutator  $[b,R_{S_\lambda}]$ is bounded on $L^p(\mathbb{R}_+,dx)$ with the operator norm
\begin{eqnarray}\label{counter}
\left\|[b, R_{S_\lambda}] \right\|_{L^p(\mathbb{R}_+,\,dx) \to L^p(\mathbb{R}_+,\,dx)} \le C\|b\|_{{\rm BMO}(\mathbb{R}_+,\,dx)}.
\end{eqnarray}

However, the BMO space associated with the Bessel operator $S_\lambda$ is the BMO$_o(\mathbb{R}_+,dx)$ as defined in \eqref{BMOo}, which is the dual of
the Hardy space $H^1_{S_\lambda}(\mathbb{R}_+,dx)$ associated with $S_\lambda$. As indicated in \cite{DDSY},
$$   {\rm BMO}_o(\mathbb{R}_+,dx) \subsetneqq  {\rm BMO}(\mathbb{R}_+,dx).  $$

Hence, we now choose
$$ b_0(x)=\log(x),\quad x>0. $$
Then it is obvious that this function $b_0 \in {\rm BMO}(\mathbb{R}_+,dx)$ but  $b_0 \not\in {\rm BMO}_o(\mathbb{R}_+,dx) $. Hence,
\begin{eqnarray*}
\left\|[b_0, R_{S_\lambda}] \right\|_{L^p(\mathbb{R}_+,\,dx) \to L^p(\mathbb{R}_+,\,dx)} \le C\|b_0\|_{{\rm BMO}(\mathbb{R}_+,\,dx)}
\end{eqnarray*}
but
$$ \|b_0\|_{{\rm BMO}_o(\mathbb{R}_+,\,dx)}=\infty. $$

\end{proof}

\medskip

\begin{proof}[\bf Proof of Corollary \ref{c weak factorization}]



Suppose $H^1_{S_\lambda}(\mathbb{R}_+,dx)$ has a weak factorization in the following form: for certain $p\in(1, \fz)$ and $f\in H^1_{S_\lambda}(\mathbb{R}_+,dx)$, there exist numbers $\{\az^k_j\}_{k,\,j}$, functions $\{g^k_j\}_{k,\,j}\subset L^p(\mathbb{R}_+,dx)$ and $\{h^k_j\}_{k,\,j}\subset L^{p'}(\mathbb{R}_+,dx)$ such that

\begin{equation*}
f=\sum_{k=1}^\fz\sum_{j=1}^\fz \az^k_j\,\Pi_{S_\lambda}\lf(g^k_j,h^k_j\r)
\end{equation*}
in $H^1_{S_\lambda}(\mathbb{R}_+,dx)$, where the operator $\Pi_{S_\lambda}$ is defined
as follows: for $g\in L^p(\mathbb{R}_+,dx)$ and $h\in L^{p'}(\mathbb{R}_+,dx)$,
\begin{equation}\label{def of pi 1}
\Pi_{{S_\lambda}}(g,h):=gR_{S_\lambda} h- h \widetilde{R_{S_\lambda}} g.
\end{equation}

We let $b_0(x)=\log(x),\  x>0$.
Then for the index $p$ above, from the inequality \eqref{counter}
we have
\begin{eqnarray*}
\left\|[b_0, R_{S_\lambda}] \right\|_{L^p(\mathbb{R}_+,\,dx) \to L^p(\mathbb{R}_+,\,dx)} \le C\|b_0\|_{{\rm BMO}(\mathbb{R}_+,\,dx)}<\infty.
\end{eqnarray*}

Now following the same proof of (2) in Theorem \ref{t-commutator character BMO} and the duality of $H^1_{S_\lambda}(\mathbb{R}_+,dx)$ with ${\rm BMO}_o(\mathbb{R}_+,dx)$, we obtain directly that
\begin{eqnarray*}
\|b_0\|_{{\rm BMO}_o(\mathbb{R}_+,\,dx)}&\le& C\left\|[b_0, R_{S_\lambda}] \right\|_{L^p(\mathbb{R}_+,\,dx)\to L^p(\mathbb{R}_+,\,dx)},
\end{eqnarray*}
which contradicts with
the fact that $$ \|b_0\|_{{\rm BMO}_o(\mathbb{R}_+,\,dx)}=\infty. $$

\end{proof}

\bigskip
\bigskip

{\bf Acknowledgments:}  The third author would like to thank Professor Richard Rochberg for a helpful discussion.
X. T. Duong is supported by ARC DP 140100649.
B. D. Wick's research supported in part by National Science Foundation
DMS grants \#0955432 and \#1500509.
D. Yang is supported by the NNSF of China (Grant No. 11571289) and the State Scholarship Fund of China (No. 201406315078)



\smallskip

Department of Mathematics, Macquarie University, NSW, 2109, Australia.

\smallskip

{\it E-mail}: \texttt{xuan.duong@mq.edu.au}

\vspace{0.3cm}



Department of Mathematics, Macquarie University, NSW, 2109, Australia.

\smallskip

{\it E-mail}: \texttt{ji.li@mq.edu.au}

\vspace{0.3cm}



Department of Mathematics, Washington University--St. Louis, St. Louis, MO 63130-4899 USA

\smallskip

{\it E-mail}: \texttt{wick@math.wustl.edu}

\vspace{0.3cm}



School of Mathematical Sciences, Xiamen University, Xiamen 361005,  China

\smallskip

{\it E-mail}: \texttt{dyyang@xmu.edu.cn }
\end{document}